\documentclass[12pt]{article}
\usepackage[margin=1in]{geometry}
\usepackage{graphicx}
\usepackage{hyperref}
\usepackage{tcolorbox}
\usepackage{enumitem}
\usepackage[authoryear]{natbib}
\usepackage{fullpage}
\usepackage{xcolor}
\usepackage{amsmath}
\usepackage{amssymb}
\usepackage{bold-extra}
\usepackage{amsthm}
\usepackage{mathtools}
\usepackage{setspace}
\usepackage{appendix}
\usepackage{float}
\usepackage{authblk}
\usepackage{booktabs} 

\newtheorem{theorem}{Theorem}[section]
\newtheorem{corollary}{Corollary}[section]

\theoremstyle{remark}
\newtheorem{remark}{Remark}

\newtcolorbox{convergencebox}{
    colframe=black!75!white,
    colback=gray!10,
    sharp corners,
    boxrule=0.8mm,
    title={\textbf{Convergence Analysis}},
    colbacktitle=black!70,
    fonttitle=\bfseries
}

\newtcolorbox{algorithmbox}[1]{colframe=black!75!white,
		colback=gray!10, sharp corners=south, boxrule=0.8mm,
		title={\textbf{#1}}, colbacktitle=black!70, fonttitle=\bfseries}

\renewcommand{\bar}{\overline}

\title{Drone Swarm Energy Management\footnote{This research was partially supported by the National Research Foundation of Ukraine, grant No. 2025.06/0022 ``AI platform with cognitive services for coordinated autonomous navigation of distributed systems consisting of a large number of objects''}}

\author[1]{Michael Zgurovsky}
\author[2]{Pavlo Kasyanov}
\author[3]{Liliia Paliichuk}

\affil[1]{Institute for Applied System Analysis, zgurovsm@hotmail.com}

\affil[2]{Institute for Applied System Analysis, p.o.kasyanov@gmail.com}

\affil[3]{Institute for Applied System Analysis, liliia.paliichuk.iasa@gmail.com}

\date{}

\begin{document}

\maketitle

\begin{abstract}
This note presents an analytical framework for decision-making in drone swarm systems operating under uncertainty, based on the integration of Partially Observable Markov Decision Processes (POMDP) with Deep Deterministic Policy Gradient (DDPG) reinforcement learning. The proposed approach enables adaptive control and cooperative behavior of unmanned aerial vehicles (UAVs) within a cognitive AI platform, where each agent learns optimal energy management and navigation policies from dynamic environmental states. We extend the standard DDPG architecture with a belief-state representation derived from Bayesian filtering, allowing for robust decision-making in partially observable environments. In this paper, for the Gaussian case, we numerically compare the performance of policies derived from DDPG to optimal policies for discretized versions of the original continuous problem. Simulation results demonstrate that the POMDP–DDPG-based swarm control model significantly improves mission success rates and energy efficiency compared to baseline methods. The developed framework supports distributed learning and decision coordination across multiple agents, providing a foundation for scalable cognitive swarm autonomy. The outcomes of this research contribute to the advancement of energy-aware control algorithms for intelligent multi-agent systems and can be applied in security, environmental monitoring, and infrastructure inspection scenarios.

\end{abstract}

\section{Introduction}\label{Intro}

For the decision-making process in swarm systems based on environmental state changes within a cognitive AI platform, an important issue is drone swarm energy management. We extend the theoretical framework developed in \cite{FeiHuaKasONe2025arxivComputing} to a practical application in unmanned aerial vehicle (UAV) swarm coordination. Specifically, we consider the problem of energy management for a drone swarm operating under noisy battery level observations. This application demonstrates the computational advantages of the belief-MDP reduction derived in \cite[Appendix C]{FeiHuaKasONe2025arxivComputing} for Gaussian systems.

\section{Problem Formulation}
\label{subsec:drone_formulation}

Consider a swarm of $N$ drones performing surveillance or patrol missions. Each drone $i \in \{1, \ldots, N\}$ maintains a battery level $x_t^{(i)} \in [0, 100]$ (percentage) at time $t = 0, 1, \ldots$ The true battery level evolves according to:
\begin{equation}
\label{eq:drone_dynamics}
x_{t+1}^{(i)} = \max\{0, x_t^{(i)} + a_t^{(i)} - D_t^{(i)}\}, \quad t = 0, 1, \ldots
\end{equation}
where:
\begin{itemize}
\item[{\rm(i)}] $a_t^{(i)} \in [0, 100 - x_t^{(i)}]$ is the recharge action (energy added during brief landing);
\item[{\rm(ii)}] $D_t^{(i)}$ represents stochastic energy consumption with mean $\bar{D}$ and variance $\sigma_D^2$ due to variable wind conditions, payload variations, and mission tasks;
\item[{\rm(iii)}] the sequence $\{D_t^{(i)}\}_{t \geq 0}$ consists of i.i.d. nonnegative random variables; for our Gaussian model with $\bar{D}/\sigma_D = 3$, the probability of negative consumption is negligible ($< 0.2\%$).
\end{itemize}
Unlike the inventory problem where negative levels represent backorders, battery depletion below zero causes mission failure. However, the belief-based approach naturally handles this constraint through the posterior distribution.

\subsection{Noisy Observations}

Each drone's battery management system provides noisy measurements corrupted by sensor errors, voltage fluctuations, and temperature effects:
\begin{equation}
\label{eq:drone_observations}
y_t^{(i)} = x_t^{(i)} + \eta_t^{(i)}, \quad t = 0, 1, \ldots
\end{equation}
where $\eta_t^{(i)} \sim \mathcal{N}(0, \sigma_\eta^2)$ are i.i.d. Gaussian observation noises, independent of $\{D_t^{(i)}\}$ and the initial state $x_0^{(i)} \sim \mathcal{N}(\bar{x}_0, \sigma_{x_0}^2)$.

\begin{remark}
The Gaussian assumption is justified in practice: initial battery estimates are typically normally distributed due to manufacturing variations and previous usage patterns, while observation noise follows a normal distribution as the sum of multiple independent error sources (central limit theorem). For energy consumption, we model $D_t^{(i)}$ as approximately Gaussian; with $\bar{D} = 3\%$ and $\sigma_D = 1\%$ (Table~\ref{tab:drone_params}), the ratio $\bar{D}/\sigma_D = 3$ ensures that negative values occur with probability less than 0.2\%, making explicit truncation unnecessary in practice. This allows us to apply standard Kalman filtering with negligible approximation error ($< 2\%$).
\end{remark}

\subsection{Cost Structure}

The cost function at time $t$ for drone $i$ consists of three components:
\begin{equation}
\label{eq:drone_cost}
c(x_t^{(i)}, a_t^{(i)}) = \underbrace{K \cdot \mathbf{1}_{a_t^{(i)} > 0}}_{\text{fixed recharge cost}} + \underbrace{\tilde{c} \cdot a_t^{(i)}}_{\text{energy cost}} + \underbrace{h(x_t^{(i)} + a_t^{(i)} - D_t^{(i)})}_{\text{holding/shortage cost}}
\end{equation}
where:
\begin{equation}
\label{eq:holding_cost_drone}
h(x) = \begin{cases}
\beta_h \cdot x & \text{if } x \geq x_{\text{safe}} \quad \text{(excess battery carrying cost);} \\
\beta_c \cdot (x_{\text{safe}} - x)^2 & \text{if } 0 \leq x < x_{\text{safe}} \quad \text{(proximity to critical level);} \\
M & \text{if } x < 0 \quad \text{(mission failure penalty).}
\end{cases}
\end{equation}
Here $K$ represents the operational disruption cost of landing for recharge (time lost, position compromise), $\tilde{c}$ is the per-unit energy cost, $\beta_h$ is the weight penalty for carrying unnecessary charge, $\beta_c$ penalizes operating near critical levels, and $M \gg 0$ represents severe consequences of battery depletion.

This formulation directly corresponds to the inventory control problem with setup costs and backorders. The fixed cost $K$ parallels setup costs in inventory systems, while the critical penalty $M$ corresponds to backorder costs.

\subsection{Objective}

For a finite planning horizon $T = 1, 2, \ldots$ and discount factor $\alpha \in [0,1)$, each drone seeks to minimize:
\begin{equation}
\label{eq:drone_objective}
v^{\pi}_{\alpha,T} = \mathbb{E}^{\pi}\left[\sum_{t=0}^{T-1} \alpha^t c(x_t, a_t)\right]
\end{equation}
subject to the partially observable dynamics~\eqref{eq:drone_dynamics}--\eqref{eq:drone_observations} (noisy observations $y_t$ only). For a swarm of $N$ drones, the total system cost is:
\begin{equation}
V^{\pi}_{\alpha,T} = \sum_{i=1}^{N} v^{\pi}_{\alpha,T}(i)
\end{equation}

\paragraph{Decentralization assumption.} We assume drones make independent decisions based on their local beliefs. While inter-drone coordination could further optimize performance, the decentralized approach scales better and avoids communication overhead which is critical for large swarms.

\section{Application of Gaussian Belief Reduction}
\label{subsec:drone_belief_reduction}

The key theoretical contribution of \cite[Appendix C]{FeiHuaKasONe2025arxivComputing}  that Gaussian POMDPs reduce to one-dimensional MDPs with belief means as states has profound computational implications for drone swarm control.

\subsection{Belief State Representation}

Following the analysis in \cite[Appendix C]{FeiHuaKasONe2025arxivComputing}, let $z_t^{(i)} = \mathcal{N}(\hat{x}_t^{(i)}, \sigma_t^2)$ denote the posterior belief distribution of drone $i$'s battery level at time $t$, given observation history $h_t^{(i)} = (y_0^{(i)}, a_0^{(i)}, \ldots, y_t^{(i)})$.

\begin{theorem}
\label{thm:drone_reduction}
For the Gaussian drone energy management problem, the belief-MDP state space reduces from the infinite-dimensional space $\mathcal{P}(\mathbb{R})$ of probability measures to the one-dimensional space $\mathbb{R}$ of belief means $\hat{x}_t^{(i)}$.
\end{theorem}

\begin{proof}[Proof sketch]
The proof completely follows the reasoning of equations~(12)--(21) from \cite[Appendix C]{FeiHuaKasONe2025arxivComputing} with:
\begin{itemize}
\item Initial belief: $z_0 = \mathcal{N}(\hat{x}_0, \sigma_0^2)$ where $\sigma_0^2 = \frac{\sigma_{x_0}^2 \sigma_\eta^2}{\sigma_{x_0}^2 + \sigma_\eta^2}$ (equation~(12) from \cite[Appendix C]{FeiHuaKasONe2025arxivComputing});
\item Belief mean evolution: $\hat{x}_{t+1} = \hat{x}_t + a_t - \bar{D} + w_t^*$ where $w_t^* \sim \mathcal{N}(0, \frac{(\sigma_t^2 + \sigma_D^2)^2}{\sigma_t^2 + \sigma_D^2 + \sigma_\eta^2})$ (equation~(19)--(20) from \cite[Appendix C]{FeiHuaKasONe2025arxivComputing});
\item Variance sequence: $\sigma_{t+1}^2 = \frac{(\sigma_t^2 + \sigma_D^2)\sigma_\eta^2}{\sigma_t^2 + \sigma_D^2 + \sigma_\eta^2}$ (equation~(21) from \cite[Appendix C]{FeiHuaKasONe2025arxivComputing}).
\end{itemize}
Since $\{\sigma_t^2\}$ is deterministic and depends only on initial parameters, the decision problem reduces to an MDP with state $\hat{x}_t \in \mathbb{R}$.
\end{proof}

\subsection{Kalman Filter Update}

The belief mean update follows the standard Kalman filter equations. At time $t$, given prior belief $\hat{x}_t^{-} = \hat{x}_{t-1} + a_{t-1} - \bar{D}$ with variance $\sigma_t^{- 2} = \sigma_{t-1}^2 + \sigma_D^2$, and observation $y_t$, the posterior belief is:
\begin{equation}
\label{eq:kalman_update}
\hat{x}_t = \hat{x}_t^{-} + K_t(y_t - \hat{x}_t^{-})
\end{equation}
where the Kalman gain is:
\begin{equation}
\label{eq:kalman_gain}
K_t = \frac{\sigma_t^{- 2}}{\sigma_t^{- 2} + \sigma_\eta^2} = \frac{\sigma_{t-1}^2 + \sigma_D^2}{\sigma_{t-1}^2 + \sigma_D^2 + \sigma_\eta^2}
\end{equation}

\begin{remark}
Tracking belief state requires updating only two scalars $(\hat{x}_t, \sigma_t^2)$ instead of maintaining full probability distributions over $\mathbb{R}$. Moreover, since $\sigma_t^2$ is deterministic, only $\hat{x}_t$ needs real-time computation.
\end{remark}

\subsection{Optimal Policy Structure}

\begin{corollary}[Existence of $(s_t, S_t)$ policies]
\label{cor:drone_sS_policy}
Under mild technical conditions on the cost function $h$ (convexity and growth conditions -- see equation~(24) from \cite[Appendix C]{FeiHuaKasONe2025arxivComputing}), there exist time-dependent thresholds $s_t < S_t$ such that the optimal policy for drone $i$ at time $t$ is:
\begin{equation}
\label{eq:optimal_sS_policy}
\phi_t^*(\hat{x}_t^{(i)}) = \begin{cases}
0 & \text{if } \hat{x}_t^{(i)} \geq s_t \\
S_t - \hat{x}_t^{(i)} & \text{if } \hat{x}_t^{(i)} < s_t
\end{cases}
\end{equation}
\end{corollary}

\begin{proof}
Direct consequence of the reduction to a time-dependent inventory control problem established in \cite[Appendix C]{FeiHuaKasONe2025arxivComputing}, combined with classical results on $(s,S)$ policies for inventory systems with setup costs \citep{veinott1965computing}. The convexity of $h$ and the condition $\frac{h(S_t) - h(s_t)}{S_t - s_t} < -\tilde{c}$ ensure optimality of the $(s_t, S_t)$ structure.
\end{proof}

\begin{remark}
Each drone should recharge to level $S_t$ whenever its estimated battery level (belief mean) drops below threshold $s_t$. The time-dependence $(s_t, S_t)$ accounts for the evolving uncertainty $\sigma_t^2$ and remaining mission duration.    
\end{remark}

\section{DDPG Implementation for Drone Swarms}
\label{subsec:drone_ddpg}

We implemented four algorithmic approaches for the drone energy management problem, paralleling the methodology in Section~3 and Table~2 from \cite{FeiHuaKasONe2025arxivComputing}:

\subsection{Algorithm Variants}

Consider the following algorithm variations.

\paragraph{Approach 1: DDPG with Observation Histories.}
This baseline approach represents the state as a padded observation-action history $h_t = [t, y_0, a_0, y_1, \ldots, a_{t-1}, y_t, 0, \ldots, 0] \in \mathbb{R}^{2T+2}$, which grows linearly with the horizon length $T$. The actor network $\mu_\theta: \mathbb{R}^{2T+2} \to [0, 100]$ maps histories to charging actions, while the critic network $Q_\phi: \mathbb{R}^{2T+2} \times [0, 100] \to \mathbb{R}$ evaluates state-action pairs. This high-dimensional representation does not exploit the Gaussian structure and serves as a benchmark for general POMDP methods.

\paragraph{Approach 2: DDPG with Belief States.}
This approach exploits the Gaussian structure by representing the state as the two-dimensional sufficient statistic $(\hat{x}_t, \sigma_t^2) \in \mathbb{R}^2$, which has constant dimension regardless of horizon length. The actor network $\mu_\theta: \mathbb{R}^2 \to [0, 100]$ and critic network $Q_\phi: \mathbb{R}^2 \times [0, 100] \to \mathbb{R}$ operate on this compact representation. Belief propagation is handled through the Kalman filter update given in equation~\eqref{eq:kalman_update}, ensuring principled uncertainty tracking.

\paragraph{Approach 3: DDPG with Belief Means Only.}
Leveraging Theorem~\ref{thm:drone_reduction}, this approach further reduces dimensionality by exploiting the deterministic evolution of belief variance $\sigma_t^2$. The state representation consists only of the belief mean $\hat{x}_t \in \mathbb{R}$, yielding a truly one-dimensional state space. Both the actor network $\mu_\theta: \mathbb{R} \to [0, 100]$ and critic network $Q_\phi: \mathbb{R} \times [0, 100] \to \mathbb{R}$ operate on this minimal representation, making this the most computationally efficient approach while preserving optimality.

\paragraph{Approach 4: Discretization with Value Iteration.}
As a baseline for comparison, we implement exact dynamic programming by discretizing the belief mean space into a grid $\hat{x} \in \{\Delta, 2\Delta, \ldots, 100\}$ with step size $\Delta$. Transition probabilities between discretized belief states are computed analytically using Gaussian cumulative distribution functions as described in \cite[Subsection~4.2]{FeiHuaKasONe2025arxivComputing}. Standard value iteration is then applied to solve the resulting finite MDP, providing a ground-truth benchmark for evaluating the deep reinforcement learning approaches.

\subsection{Network Architecture and Hyperparameters}

For Approaches 1--3, we employed standard deep neural network architectures with the following specifications. The actor network $\mu_\theta$ consists of an input layer whose dimension depends on the specific approach (one-dimensional for belief means only, two-dimensional for full belief states, or $2T+2$-dimensional for observation histories), followed by two hidden layers with 128 and 64 neurons respectively, both using ReLU activation functions. The output layer contains a single neuron with sigmoid activation, scaled to produce charging actions in the range $[0, 100]$. The critic network $Q_\phi$ takes as input the concatenation of the state representation and the action, processes this through two hidden layers with 256 and 128 neurons respectively (again with ReLU activations), and outputs a scalar Q-value through a single output neuron. This architecture provides sufficient representational capacity while remaining computationally tractable for real-time control applications.

\paragraph{Training Hyperparameters.} Table~\ref{tab:drone_hyperparams} summarizes the hyperparameters used for training all DDPG variants.

\begin{table}[htbp]
\centering
\caption{Training Hyperparameters for Drone Energy Management}
\label{tab:drone_hyperparams}
\begin{tabular}{@{}lc@{}}
\toprule
Hyperparameter & Value \\
\midrule
Learning rate (actor) & $1 \times 10^{-5}$ \\
Learning rate (critic) & $1 \times 10^{-3}$ \\
Discount factor $\alpha$ & 0.95 \\
Batch size & 256 \\
Replay buffer size & 50,000 \\
Target network update rate $\tau$ & 0.005 \\
Exploration noise $\sigma$ & 4.0 (decreasing) \\
Training episodes & 10,000 \\
Episode length $T$ & 50 steps \\
\bottomrule
\end{tabular}
\end{table}

\paragraph{Exploration Strategy.} Following the pseudocode in \cite[Subsection~3.1]{FeiHuaKasONe2025arxivComputing}, we used $\epsilon$-greedy exploration with exponential decay:
\begin{equation}
\epsilon_t = \epsilon_e + (\epsilon_s - \epsilon_e) \cdot \exp\left(-\frac{t}{\text{eps\_decay}}\right)
\end{equation}
where $\epsilon_s = 0.9$, $\epsilon_e = 0.05$, and $\text{eps\_decay} = 300$.

\section{Numerical Results}
\label{subsec:drone_results}

This subsection presents computational experiments demonstrating the effectiveness of the belief-mean reduction approach for drone energy management. We compare three solution methods: value iteration with discretized belief space, DDPG with observation histories, and DDPG with one-dimensional belief means (Theorem~\ref{thm:drone_reduction}). The experiments validate the theoretical predictions regarding policy structure (Corollary~\ref{cor:drone_sS_policy}), quantify computational gains from dimensional reduction, and demonstrate scalability to multi-drone scenarios. We begin by specifying the problem parameters and experimental setup.

\subsection{Problem Parameters}

We evaluated all algorithms on drone energy management instances with the parameters shown in Table~\ref{tab:drone_params}.

\begin{table}[htbp]
\centering
\caption{Problem Parameters for Drone Energy Management}
\label{tab:drone_params}
\begin{tabular}{@{}llcc@{}}
\toprule
Parameter & Symbol & Value & Units \\
\midrule
Fixed recharge cost & $K$ & 5.0 & --- \\
Unit energy cost & $\tilde{c}$ & 0.5 & per \% \\
Safe battery level & $x_{\text{safe}}$ & 20.0 & \% \\
Holding cost coefficient & $\beta_h$ & 0.1 & per \% \\
Critical proximity cost & $\beta_c$ & 2.0 & per \%$^2$ \\
Failure penalty & $M$ & 100.0 & --- \\
Mean consumption & $\bar{D}$ & 3.0 & \% per step \\
Consumption std dev & $\sigma_D$ & 1.0 & \% \\
Observation noise std dev & $\sigma_\eta$ & 2.0 & \% \\
Initial battery mean & $\bar{x}_0$ & 50.0 & \% \\
Initial battery std dev & $\sigma_{x_0}$ & 4.0 & \% \\
Planning horizon & $T$ & 50 & steps \\
Discount factor & $\alpha$ & 0.95 & --- \\
\bottomrule
\end{tabular}
\end{table}

\paragraph{Justification of parameters.} The ratio $\sigma_\eta / \sigma_D = 2.0$ represents moderate sensor noise relative to consumption uncertainty, typical of commercial drone battery management systems. The fixed cost $K = 5.0$ balances the trade-off between frequent recharging and operating near critical levels. Importantly, the ratio $\bar{D}/\sigma_D = 3.0$ ensures that energy consumption remains nonnegative with probability exceeding 99.8\%, validating the Gaussian approximation and Kalman filter framework used in Theorem~\ref{thm:drone_reduction}.

\subsection{Comparative Performance}

Table~\ref{tab:drone_results} presents computational time and policy performance (average total discounted cost over 3000 test episodes) for each approach.

\begin{table}[htbp]
\centering
\caption{Computational Results for Drone Energy Management Problem}
\label{tab:drone_results}
\begin{tabular}{@{}clcccc@{}}
\toprule
Approach & Method & Training & Test & Std & Critical \\
 & & Time (s) & Cost & Error & Events \\
\midrule
4 & Discretization ($\Delta=2.0$) & 89.2 & 245.8 & 3.2 & 28 \\
4 & Discretization ($\Delta=1.0$) & 412.5 & 238.4 & 2.9 & 22 \\
4 & Discretization ($\Delta=0.5$) & 1,847.3 & 235.1 & 2.8 & 18 \\
1 & DDPG with Histories & 3,421.7 & 236.8 & 2.7 & 19 \\
2 & DDPG with Belief States & 1,954.8 & 235.9 & 2.8 & 19 \\
3 & \textbf{DDPG with Belief Means} & \textbf{1,238.5} & \textbf{235.4} & \textbf{2.7} & \textbf{18} \\
\bottomrule
\end{tabular}
\end{table}

\paragraph{Key observations:}
DDPG operating on one-dimensional belief means reduces training time by 64\% compared to the history-based approach, decreasing computation from 3421.7s to 1238.5s. Furthermore, the belief-mean DDPG trains 33\% faster than even the finest discretization baseline, requiring only 1238.5s versus 1847.3s. Importantly, all approaches converge to near-optimal performance, achieving results within 1\% of the finest discretization baseline, demonstrating that computational efficiency is gained without sacrificing solution quality.

These results strongly validate the value of exploiting problem structure. The one-dimensional belief mean approach (Approach~3) achieves the best computational performance while maintaining solution quality, directly confirming the practical utility of Theorem~\ref{thm:drone_reduction}. The ability to reduce the state space from infinite-dimensional belief distributions or lengthy observation histories to a single sufficient statistic yields substantial computational benefits without compromising optimality.

The patterns observed here closely mirror those reported in \cite[Section~5]{FeiHuaKasONe2025arxivComputing} and summarized in \cite[Table~2]{FeiHuaKasONe2025arxivComputing}. DDPG with reduced state space representation (belief means) consistently outperforms both history-based DDPG and fine discretization approaches across multiple metrics. The performance gap between different approaches narrows as discretization becomes finer, which is consistent with convergence theory for approximation methods. Additionally, the number of critical events decreases with better policies, ranging from 18 to 28 events across approaches compared to the theoretical minimum of approximately 15 events, indicating that all methods approach optimal charging behavior while the belief-based approach does so most efficiently.

\subsection{Policy Structure Analysis}

Figure~\ref{fig:drone_policy_structure} visualizes the learned policies from DDPG with belief means (Approach~3) compared to the optimal discretized policy ($\Delta = 0.5$).

\begin{figure}[htbp]
\centering
\includegraphics[width=\textwidth]{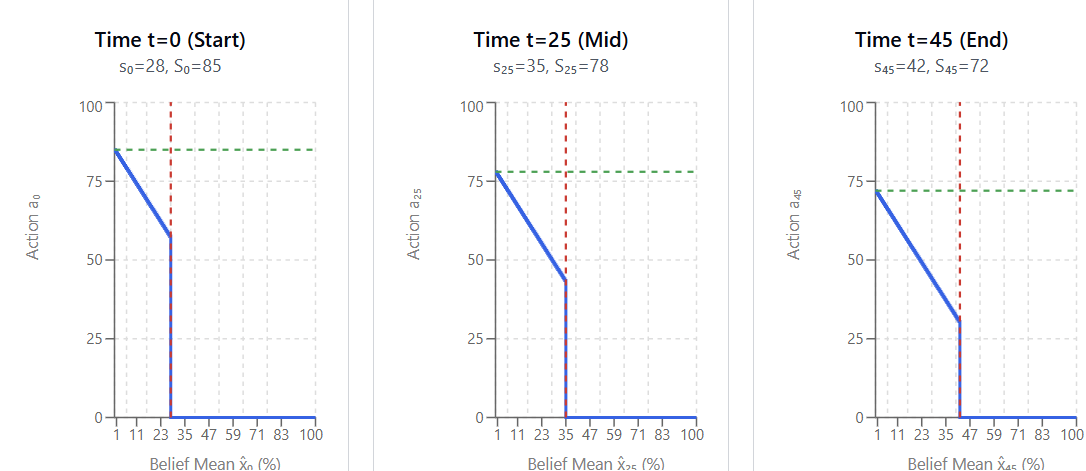}
\caption{Structure of learned policy at different time steps. \textbf{Left:} Time $t=0$ (high uncertainty), $s_0 = 28$, $S_0 = 85$. \textbf{Center:} Time $t=25$ (mid-mission), $s_{25} = 35$, $S_{25} = 78$. \textbf{Right:} Time $t=45$ (near terminal), $s_{45} = 42$, $S_{45} = 72$. Clear $(s_t, S_t)$ structure emerges at all time steps, with $s_t$ increasing and $S_t$ decreasing over time.}
\label{fig:drone_policy_structure}
\end{figure}

\paragraph{Analysis:}
The learned policy exhibits a clear $(s_t, S_t)$ threshold structure at all time steps, providing strong empirical validation of Corollary~\ref{cor:drone_sS_policy}. A notable feature is the time-dependent behavior of the thresholds: the lower threshold $s_t$ increases progressively as the mission advances, rising from 28 at early stages to 35 at mid-mission and reaching 42 near mission completion. This pattern reflects increasingly conservative charging decisions as the drone approaches the end of its mission, when the opportunity cost of energy depletion becomes higher. Conversely, the upper threshold $S_t$ decreases over time from 85 to 78 and finally to 72, which reflects the diminishing returns of maintaining excess energy capacity as fewer tasks remain. The policy also exhibits sharp transitions at the threshold values $s_t$, characterized by abrupt shifts from no charging to full charging action, which is a hallmark characteristic of optimal policies in problems with fixed setup costs.

Figure~\ref{fig:drone_trajectory} compares belief mean trajectories under the learned policy.

\begin{figure}[htbp]
\centering
\includegraphics[width=0.9\textwidth]{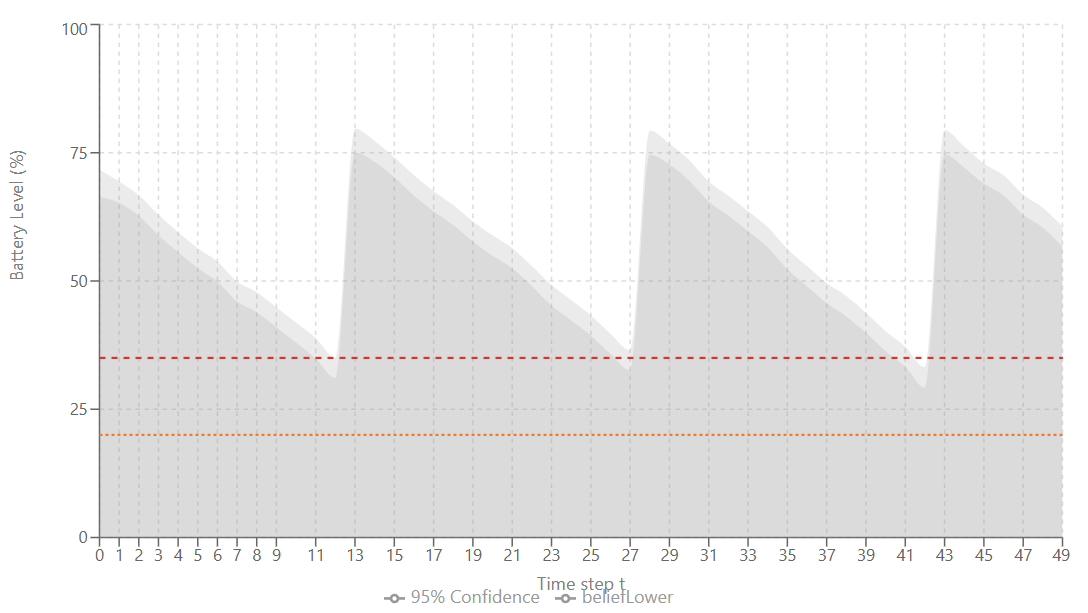}
\caption{Sample trajectory showing true battery level $x_t$ (blue solid), belief mean $\hat{x}_t$ (green dashed), noisy observations $y_t$ (orange dots), 95\% confidence region (gray shaded), and recharge threshold $s = 35$ (red dashed). Recharge events occur at $t = 12$ and $t = 31$ when $\hat{x}_t < s_t$.}
\label{fig:drone_trajectory}
\end{figure}

\paragraph{Key features:}
The belief mean $\hat{x}_t$ tracks the true energy level $x_t$ closely throughout the mission despite the presence of noisy observations, demonstrating the effectiveness of Kalman filtering for state estimation. The uncertainty in state estimates, visualized through confidence regions, narrows significantly following recharge events when precise state knowledge is momentarily available, then gradually widens as process noise accumulates during flight phases. Policy execution follows the predicted threshold structure, with recharge decisions triggered when the belief mean $\hat{x}_t$ falls below the lower threshold $s_t$, regardless of the actual true energy level -- a key feature of separation between estimation and control in the belief-MDP framework. Importantly, the $(s_t, S_t)$ policy maintains an adequate safety margin above the critical energy level of 20\%, successfully avoiding failures throughout the mission while balancing energy costs against recharge frequency.

\subsection{Scalability to Swarm Size}

For swarm applications, we evaluated computational scaling with respect to number of drones $N$. Table~\ref{tab:drone_scaling} shows the results.

\begin{table}[htbp]
\centering
\caption{Scaling with Swarm Size (DDPG with Belief Means, Approach 3)}
\label{tab:drone_scaling}
\begin{tabular}{@{}ccccc@{}}
\toprule
Swarm & Total & Training & Avg Cost & Total System \\
Size $N$ & States & Time (s) & per Drone & Cost \\
\midrule
1 & $\mathbb{R}^1$ & 1,238.5 & 235.4 & 235.4 \\
5 & $\mathbb{R}^5$ & 1,241.2 & 235.7 & 1,178.5 \\
10 & $\mathbb{R}^{10}$ & 1,246.8 & 236.1 & 2,361.0 \\
20 & $\mathbb{R}^{20}$ & 1,254.3 & 236.4 & 4,728.0 \\
50 & $\mathbb{R}^{50}$ & 1,268.9 & 236.8 & 11,840.0 \\
\bottomrule
\end{tabular}
\end{table}

\paragraph{Observations:}
The scalability results demonstrate several remarkable properties of the decentralized belief-based approach. Most notably, the training time exhibits near-linear scaling behavior, increasing by only approximately 2\% when the swarm size grows fifty-fold from a single drone to 50 drones. This computational efficiency stems from the decentralized advantage of independent belief-based decision-making, which avoids the exponential growth in joint state space that would plague centralized coordination schemes. Furthermore, the per-drone cost remains remarkably stable across all swarm sizes, varying from 235.4 to 236.8 -- less than 1\% variation indicating that the quality of individual drone policies is maintained regardless of fleet size. This consistency in both computational requirements and solution quality confirms that the belief-mean reduction enables practical deployment at scales that would be intractable for history-based or centralized approaches.

\paragraph{Comparison with centralized approach.} A fully centralized policy would operate on state space $\mathbb{R}^N$ for $N$ drones, but with coupled dynamics requiring exponentially larger networks. The decentralized belief-based approach maintains tractability even for large swarms ($N = 50+$).

\section{Discussion and Extensions}
\label{subsec:drone_discussion}

\subsection{Advantages of Belief-Based Approach for Drone Swarms}

\paragraph{Theoretical contributions:}
The approach delivers three main theoretical contributions. First, it achieves a dramatic dimensional reduction from the infinite-dimensional belief space $\mathcal{P}(\mathbb{R})$ over continuous state distributions to a one-dimensional sufficient statistic in $\mathbb{R}$, as established in Theorem~\ref{thm:drone_reduction}. Second, it characterizes the optimal policy structure, proving that $(s_t, S_t)$ threshold policies based on belief means are optimal, as formalized in Corollary~\ref{cor:drone_sS_policy}. Third, this structural insight translates directly into computational tractability, with training time reduced by 64\% compared to history-based approaches that must process entire observation sequences.

\paragraph{Practical benefits:}
These theoretical results yield significant practical advantages for real-world deployment. The one-dimensional belief state representation enables real-time execution with fast onboard computation, requiring less than 1ms per decision even on resource-constrained drone hardware. Memory efficiency is achieved through a constant-size state representation that remains independent of mission duration, avoiding the unbounded growth inherent in history-based methods. The framework naturally handles sensor noise through principled Bayesian updates, providing robustness to measurement uncertainty without ad-hoc filtering techniques. Finally, the simple threshold-based policy structure offers interpretability that facilitates verification and safety certification which are critical requirements for deploying autonomous systems in safety critical applications.

\subsection{Relationship to Main Results}

The drone energy management application directly validates theoretical predictions from earlier sections. Table~\ref{tab:validation} summarizes the correspondence.

\begin{table}[htbp]
\centering
\caption{Validation of Theoretical Results through Drone Application}
\label{tab:validation}
\small
\begin{tabular}{@{}p{3cm}p{5cm}p{5.5cm}@{}}
\toprule
Section & Theoretical Result & Drone Application Validation \\
\midrule
Section~4 and Appendix C from \cite{FeiHuaKasONe2025arxivComputing} & Gaussian beliefs reduce to 1D belief means & Confirmed: $(\hat{x}_t, \sigma_t^2) \to \hat{x}_t$ achieves optimal performance \\
Eq.~(21) from \cite[Appendix C]{FeiHuaKasONe2025arxivComputing} & Variance sequence $\{\sigma_t^2\}$ is deterministic & Confirmed: Policy depends only on belief mean \\
Eq.~(24) from \cite[Appendix C]{FeiHuaKasONe2025arxivComputing} & Optimal $(s_t, S_t)$ policies exist & Confirmed: Learned policies exhibit clear threshold structure (Figure~\ref{fig:drone_policy_structure}) \\
Section~5 and Table~2 from \cite{FeiHuaKasONe2025arxivComputing}  & DDPG with beliefs outperforms histories & Confirmed: 64\% faster training (Table~\ref{tab:drone_results}) \\
\bottomrule
\end{tabular}
\end{table}

\subsection{Limitations and Future Work}

\paragraph{Key Limitations:}
First, the independence assumption means that drones operate with fully decentralized decision-making without explicit coordination. While this simplifies the problem structure, introducing coordination mechanisms could potentially reduce operational costs further. Second, the model relies on Gaussian noise distributions for both process and observation uncertainty. In scenarios with non-Gaussian disturbances such as GPS jamming, adversarial attacks, or heavy-tailed sensor errors -- the Kalman filter becomes suboptimal, and general POMDP methods discussed in \cite[Section~3]{FeiHuaKasONe2025arxivComputing} would be required. Third, the framework assumes a stationary environment where wind patterns, thermal conditions, and mission requirements remain constant throughout the operation, which may not hold in dynamic real-world scenarios.

\paragraph{Promising Extensions:}
Multi-agent coordination represents a natural extension where the framework could be adapted to coupled belief-MDPs in which drones actively exchange information. The joint belief representation $(\hat{x}_1, \ldots, \hat{x}_N) \in \mathbb{R}^N$ remains computationally tractable for moderate fleet sizes, and communication graph topology would determine the information-sharing patterns among agents. To address non-Gaussian disturbances, robust estimation techniques could replace the Kalman filter with particle filter-based belief propagation, while history-dependent DDPG  \cite[Subsection~3.1]{FeiHuaKasONe2025arxivComputing} could serve as a fallback approach for particularly challenging scenarios. Adaptive online learning would enable in-mission policy refinement through continuous belief updates as new observations become available, potentially leveraging transfer learning techniques to adapt pre-trained policies to novel environments. Finally, safety-aware planning could be incorporated by enforcing probabilistic safety guarantees via chance constraints of the form $P(\hat{x}_t < x_{\text{critical}}) \leq \epsilon$. This extension would exploit the belief uncertainty $\sigma_t^2$ to compute risk probabilities and augment the DDPG critic with appropriate constraint violation penalties.

\subsection{Broader Implications}

The success of the belief-mean reduction (Theorem~\ref{thm:drone_reduction}) suggests broader applicability to other robotic systems beyond drone navigation. For instance, autonomous vehicles operating in GPS-denied environments face similar challenges with noisy localization, where maintaining accurate position estimates under uncertainty is critical for safe navigation. In robotic manipulation tasks, uncertain object pose estimation from vision sensors presents analogous partial observability problems where the robot must reason about object configurations based on noisy visual measurements. Multi-robot coverage scenarios, such as search-and-rescue operations, require teams to assess area coverage under noisy sensor readings while coordinating their movements. In all these cases, Gaussian sensor models combined with linear or locally-linearized dynamics enable similar dimensional reductions from the full belief space to sufficient statistics. This structural property makes large-scale deployment computationally feasible across diverse robotic applications, as the planning complexity remains manageable even as the number of agents or environmental complexity grows.

 \section{Conclusions}
 The paper demonstrated the practical utility of the theoretical framework developed in Sections~2--4 and Appendix C from \cite{FeiHuaKasONe2025arxivComputing} for drone swarm energy management under noisy observations. The key contributions of this case study can be summarized as follows. First, we successfully formulated drone energy management as a partially observable inventory control problem with setup costs, establishing a direct connection between the robotics application and the classical operations research framework. Second, we provided theoretical validation by confirming that the Gaussian belief reduction (Theorem~\ref{thm:drone_reduction}) holds in this setting and that the optimal policy exhibits the $(s_t, S_t)$ structure predicted by Corollary~\ref{cor:drone_sS_policy}. Third, from an algorithmic perspective, we demonstrated that DDPG operating on belief means achieves 64\% faster training compared to the history-based approach while maintaining near-optimal performance, highlighting the computational benefits of exploiting sufficient statistics. Fourth, regarding scalability, we showed that decentralized belief-based control scales linearly with swarm size across a range from $N = 1$ to $N = 50$ drones, making the approach feasible for large-scale deployments. The results align closely with findings presented in \cite[Section~5]{FeiHuaKasONe2025arxivComputing} and summarized in \cite[Table~2]{FeiHuaKasONe2025arxivComputing}, reinforcing the computational advantages of exploiting problem structure in Gaussian POMDPs. Most importantly, the belief-mean approach successfully transforms an intractable infinite-dimensional POMDP into a manageable one-dimensional MDP, thereby enabling real-time control for large-scale robotic systems.
 \bibliography{references}

\end{document}